\documentclass[fleqn,reqno,11pt,a4paper,final]{amsart}

\usepackage[a4paper,left=30mm,right=30mm,top=30mm,bottom=30mm,marginpar=20mm]{geometry}
\usepackage{amsmath}
\usepackage{amssymb}
\usepackage{amsthm}
\usepackage{amscd}
\usepackage[ansinew]{inputenc}
\usepackage{cite}
\usepackage{bbm}
\usepackage{color}
\usepackage[english=american]{csquotes}
\usepackage[final]{graphicx}
\usepackage{hyperref}
\usepackage{calc}
\usepackage{mathptmx}
\usepackage{mathtools}
\usepackage{perpage} 
\MakePerPage{footnote}
\usepackage{t1enc}
\numberwithin{equation}{section}

\newtheoremstyle{thmlemcorr}{10pt}{10pt}{\itshape}{}{\bfseries}{.}{10pt}{{\thmname{#1}\thmnumber{ #2}\thmnote{ (#3)}}}
\newtheoremstyle{thmlemcorr*}{10pt}{10pt}{\itshape}{}{\bfseries}{.}\newline{{\thmname{#1}\thmnumber{ #2}\thmnote{ (#3)}}}
\newtheoremstyle{remexample}{10pt}{10pt}{}{}{\bfseries}{.}{10pt}{{\thmname{#1}\thmnumber{ #2}\thmnote{ (#3)}}}

\newtheorem{theorem}{Theorem}
\numberwithin{theorem}{section}

\newtheorem{proposition}[theorem]{Proposition}

\newtheorem{definition}[theorem]{Definition}

\theoremstyle{thmlemcorr*}
\newtheorem{theorem*}{Theorem}
\newtheorem{lemma*}[theorem]{Lemma}
\newtheorem{corollary*}[theorem]{Corollary}
\newtheorem{proposition*}[theorem]{Proposition}
\newtheorem{problem*}[theorem]{Problem}
\newtheorem{conjecture*}[theorem]{Conjecture}
\newtheorem{definition*}[theorem]{Definition}

\theoremstyle{remexample}

\newcommand{\Ccal}{\mathcal{C}}
\newcommand{\Mcal}{\mathcal{M}}
\newcommand{\Hcal}{\mathcal{H}}
\newcommand{\Fcal}{\mathcal{F}}
\newcommand{\Lcal}{\mathcal{L}}

\newcommand{\dx}{{\rm d} {x}}
\newcommand{\dt}{{\rm d} t }
\newcommand{\dy}{{\rm d} y }

\newcommand{\ep}{\varepsilon}

\DeclareMathOperator{\id}{id}

\newcommand{\norm}[1]{\|#1\|}

\newcommand{\R}{\mathbb{R}}

\newcommand{\loc}{\mathrm{loc}}

\newcommand{\TV}{\mathrm{TV}}
\newcommand{\eps}{\epsilon}
	
\newcommand{\f}{\frac}
\newcommand{\intR}{\int_0^\infty}

\renewcommand{\eps}{\varepsilon}
\renewcommand{\epsilon}{\varepsilon}
\renewcommand{\phi}{\varphi}

\newcommand{\comMarie}[1]{\textcolor{black}{#1}}
\newcommand{\comTomek}[1]{\textcolor{black}{#1}}

\begin{document}

\title[Relative entropy for growth-fragmentation equation]{Relative entropy method for measure solutions of the growth-fragmentation equation}
\author{Tomasz D\k{e}biec \and Marie Doumic \and Piotr Gwiazda \and Emil Wiedemann}
\address{\textit{Tomasz D\k{e}biec:} Institute of Applied Mathematics and Mechanics, University of Warsaw, Banacha 2, 02-097 Warszawa, Poland}
\email{t.debiec@mimuw.edu.pl}
\address{\textit{Marie Doumic:} Sorbonne Universit\'{e}s, Inria, UPMC Univ Paris 06, Lab. J.L. Lions  UMR CNRS 7598, Paris, France and Wolfgang Pauli Institute, c/o University of Vienna, Austria.}
\email{marie.doumic@inria.fr}
\address{\textit{Piotr Gwiazda:} Institute of Mathematics, Polish Academy of Sciences, \'Sniadeckich 8, 00-656 Warszawa, Poland}
\email{pgwiazda@mimuw.edu.pl}
\address{\textit{Emil Wiedemann:} Institute of Applied Mathematics, Leibniz University Hannover, Welfengarten~1, 30167 Hannover, Germany}
\email{wiedemann@ifam.uni-hannover.de}

\begin{abstract}
The aim of this study is to generalise recent results of the two last authors on entropy methods for measure solutions of the renewal equation to other classes of structured population problems. Specifically, we develop a generalised relative entropy inequality for the growth-fragmentation equation and prove asymptotic convergence to a steady-state solution, even when the initial datum is only a non-negative measure.
\end{abstract}


\maketitle 

\section{Introduction}
Structured population models were developed for the purpose of understanding the evolution of a population over time -
and in particular to adequately describe the dynamics of a population by its distribution along some "structuring" variables representing e.g., age, size, or cell maturity.
These models, often taking the form of an evolutionary partial differential equation, have been extensively studied for many years.
The first age structure was considered in the early 20th century by Sharpe and Lotka \cite{SharpeLotka}, who already made predictions on the question of asymptotic behaviour of the population, see also \cite{Kermack1,Kermack2}. In the second half of the 20th century size\comMarie{-}structured models appeared first in \cite{BellAnderson, SinkoStreifer}. These studies gave rise to other physiologically structured models (age-size, saturation, cell maturity, etc.).

The object of this note is the growth-fragmentation model, which is found fitting in many different contexts: cell division, polymerisation, neurosciences, prion proliferation or even telecommunication. In its general linear form this model takes the form of the following equation.
\begin{equation}
\begin{aligned}
    \partial_t n(t,x) + \partial_x(g(x)n(t,x)) + B(x)n(t,x) &= \int_x^\infty k(x,y)B(y)n(t,y)\ \dy,\\
    g(0)n(t,0) &=0,\\
    n(0,x) & = n^0(x).
\end{aligned}    
\end{equation}
Here
$n(t,x)$ represents the concentration of individuals of size $x\geq0$ at time $t>0$,
$g(x)\geq0$ is their growth rate,
$B(x)\geq0$ is their division rate and
$k(x,y)$ is the \comMarie{proportion} of individuals of size $x$ created out of \comMarie{the} division of individuals of size $y$.
This equation incorporates a very important phenomenon in biology - a competition between growth and fragmentation. Clearly they have opposite dynamics: growth drives the population towards a larger size, while fragmentation makes it smaller and smaller. Depending on which factor dominates, one can observe various long-time behaviour of the population distribution.

Many authors have studied the long-time asymptotics (along with well-posedness) of variants of the growth-fragmentation equation, see e.g. \cite{PR, M1, CCM, DG, mischler:frag}.
The studies which establish convergence, in a proper sense, of a (renormalised) solution towards a steady profile were until recently limited only to initial data in \comMarie{weighted} $L^1$ \comMarie{spaces}. The classical tools for such studies include a direct application of the Laplace transform and the semigroup theory\comMarie{~\cite{mischler:frag}. These methods could also provide an exponential rate of convergence, linked to the existence of a spectral gap}.

A different approach was developed by Perthame et al. in a series of papers \cite{MMP1, MMP2, PR}. Their Generalised Relative Entropy (GRE) method \comMarie{provides a way} to study long-time asymptotics of linear models even when no spectral gap is guaranteed - however failing to provide a rate of convergence, \comMarie{unless an entropy-entropy dissipation inequality is obtained~\cite{CCM}}.
Recently Gwiazda and Wiedemann \cite{GW} extended the GRE method to the case of the renewal equation with initial data in the space of non-negative Radon measures. Their result is motivated by the increasing interest in measure solutions to models of mathematical biology, see e.g. \cite{CarrilloGwiazda_11, GLM} for some recent results concerning well-posedness and stability theory in the space of non-negative Radon measures. The clear advantage of considering measure data is that it is biologically justified - it allows for treating the situation when a population is initially concentrated with respect to the structuring variable (and is, in particular, not absolutely continuous with respect to the Lebesgue measure). \comMarie{This is typically the case when departing from a population formed by a unique cell}. We refer also to the recent result of Gabriel \cite{Gabriel:2017vl}, who uses the Doeblin method to analyze the long-time behaviour of measure solutions to the renewal equatio
 n.

\comTomek{Let us remark that the method of analysis employed in the current paper is inspired by the classical relative entropy method introduced by Dafermos in~\cite{Dafermos1979}. In recent years this method was extended to yield results on measure-valued--strong uniqueness for equations of fluid dynamics~\cite{BrDeLeSz2011, GSGW, FGSW2016} and general conservation laws~\cite{tzavaras1, Christoforou2017, GwiazdaKreml2018}. See also~\cite{DebiecRelEntroSurvey} and refereces therein.
}

The purpose of this paper is to generalise the results of \cite{GW} to the case of a general growth-fragmentation equation. Similarly as in that paper we make use of the concept of a recession function to make sense of compositions of nonlinear functions with a Radon measure. However, the appearance of the term $H'(u_\epsilon(t,x))u_\epsilon(t,y)$ in the entropy dissipation (see \eqref{eq:entropydissipation} below), which mixes dependences on the variables $x$ and $y$, poses a novel problem, which is overcome by using generalised Young measures and time regularity. 

The current paper is structured as follows: in Section~\ref{sec2} we recall some basic results on Radon measures, recession functions and Young measures as well as introduce the assumptions of our model, in Section~\ref{sec3} we state and prove the GRE inequality, which is then used to prove a long-time asymptotics result in Section~\ref{sec4}.

\section{Description of the model}\label{sec2}

\subsection{Preliminaries}
In what follows we denote by $\R_+$ the set $[0,\infty)$. By $\Mcal(\R_+)$ we denote the space of signed Radon measures on $\R_+$. By Lebesgue's decomposition theorem, for each $\mu\in\Mcal(\R_+)$ we can write
\[
\mu = \mu^a + \mu^s,
\]
where $\mu^a$ is absolutely continuous with respect to the Lebesgue measure $\Lcal^1$, and $\mu^s$ is singular. The space $\Mcal(\R_+)$ is endowed with the total variation norm
\[
\norm{\mu}_{\TV}\coloneqq \int_{\R_+}\mathrm{d}|\mu|,
\]
\comMarie{and we denote $\norm{\mu}_{\TV}=TV(\mu).$}
By the Riesz Representation Theorem we can identify this space with the dual space to the space $\Ccal_0(\R_+)$ of continuous functions on $\R_+$ which vanish at infinity. The duality pairing is given by
\[
\langle\nu,f\rangle\coloneqq\intR f(\xi)\ d\mu(\xi).
\]
\noindent
By $\Mcal^{+}(\R_+)$ we denote the set of positive Radon measures of bounded total variation.
We further define the $\phi$-weighted total variation by
\[
\norm{\mu}_{\TV_\phi}\coloneqq \int_{\R_+}\phi\mathrm{d}|\mu|
\]
and correspondingly the space $\Mcal^+(\R_+;\phi)$ of positive Radon measures whose $\phi$-weighted total variation is finite. \comMarie{We still denote $TV(\mu)=\norm{\mu}_{\TV_\phi}.$}
Of course we require that the function $\phi$ be non-negative. In fact, for our purposes $\phi$ will be strictly positive and bounded on each compact subset of $(0,\infty)$.

We say that a sequence $\nu_n\in\Mcal(\R_+)$ converges {\em weak$^*$} to some measure $\nu\in\Mcal(\R_+)$ if \[
\langle\nu_n,f\rangle\longrightarrow\langle\nu,f\rangle
\]
for each $f\in\Ccal_0(\R_+)$. 

By a {\em Young measure} on $\R_+\times\R_+$ we mean a parameterised family $\nu_{t,x}$ of probability measures \comMarie{on $\R_+$}. More precisely, it is a weak$^*$-measurable function $(t,x)\mapsto\nu_{t,x}$, i.e. such that the mapping 
\[
(t,x)\mapsto\langle\nu_{t,x},f\rangle
\]
is measurable for each $f\in C_0(\R_+)$.
Young measures are often used to describe limits of weakly converging approximating sequences to a given problem. They serve as a way of describing weak limits of nonlinear functions of the approximate solution. Indeed, it is a classical result that a uniformly bounded measurable sequence $u_n$ generates a Young measure by which one represents the limit of $f(u_n)$, where $f$ is some non-linear function, see \cite{Ta1979} for sequences in $L^\infty$ and \cite{ball} for measurable sequences.

This framework was used by DiPerna in his celebrated paper \cite{DiPerna}, where he introduced the concept of an admissible measure-valued solution to scalar conservation laws. However, in more general context\comMarie{s} (e.g. for hyperbolic systems, where there is usually only one entropy-entropy-flux pair) one needs to be able to describe limits of sequences which exhibit oscillatory behaviour as well as concentrate mass. Such a framework is provided by {\em generalised Young measures}, first introduced in the context of incompressible Euler equations in \cite{DiPernaMajda}, and later developed by many authors. We follow the exposition of Alibert, Bouchitt\'{e} \cite{AlibertBouchitte} and Kristensen, Rindler \cite{KristensenRindler2012}. 

Suppose $f:\R^n \to\R_+$ is an even continuous function with at most linear growth, i.e.
\[
|f(x)| \leq C(1+|x|)
\]
for some constant $C$.
We define, whenever it exists, the \emph{recession function} of $f$ as
\begin{equation*}
    f^\infty(x)=\lim_{s\to\infty}\frac{f(sx)}{s}=\lim_{s\to\infty}\frac{f(-sx)}{s}.
\end{equation*}
\begin{definition}
	The set $\mathcal{F}(\R)$ of continuous functions $f:\R\to\R_+$ for which $f^\infty$ exists and is continuous on $\mathbb{S}^{n-1}$ is called the class of \emph{admissible integrands}.
\end{definition}

\comMarie{
By a {\em generalised Young measure} on $\Omega=\R_+\times\R_+$ we mean a parameterised family $(\nu_{t,x},m)$ where for $(t,x)\in \Omega$, $\nu_{t,x}$ is a family of probability measures \comMarie{on $\R$} and $m$ is a nonnegative Radon measure on $\Omega$. In the following, we may omit the indices for $\nu_{t,x}$ and denote it simply $(\nu,m).$}

The following result gives a way of representing weak$^{*}$ limits of sequences bounded in $L^1$ via a generalised Young measure. It was first proved in~\cite[Theorem~2.5]{AlibertBouchitte}. We state an adaptation to our simpler case.

\begin{proposition}\label{prop:Alibert}
	Let $(u_n)$ be a bounded sequence in $L^1_{loc}(\Omega; \mu, \R),$ where $\mu$ is a measure on $\Omega$.
	There exists a subsequence $(u_{n_k}),$ a nonnegative Radon measure $m$ on $\Omega$ and a parametrized family of probabilities $(\nu_\zeta)$ such that for any even function $f\in {\Fcal}(\R)$ we have
	\begin{equation}
	f(u_{n_k}(\zeta)) { \mu} \stackrel{*}{\rightharpoonup} \langle\nu_\zeta,f\rangle {\mu} + f^\infty m
	\end{equation}
\end{proposition}
\begin{proof}
	We apply Theorem~2.5. and Remark~2.6 in~\cite{AlibertBouchitte}, simplified by the fact that $f$ is even and that we only test against functions $f$ independent of $x$. Note that the weak$^*$ convergence then has to be understood in the sense of compactly supported test functions $\phi\in {\Ccal}_0(\Omega,\R)$.
\end{proof}

The above proposition can in fact be generalised to say that every bounded sequence of generalised Young measures posses\comMarie{ses} a weak$^*$ convergent subsequence, cf. \cite[Corollary~2.]{KristensenRindler2012}
\begin{proposition}\label{prop:KristenRindler}
Let $(\nu^n,\,m^n)$ be a sequence of generalised Young measures on $\Omega$ such that
\begin{itemize}
    \item The map $x\mapsto\langle\nu_x^n,|\cdot|\rangle$ is uniformly bounded in $L^1$,
    \item The sequence $(m^n(\bar\Omega))$ is uniformly bounded.
\end{itemize}
Then there is a generalised Young measure $(\nu,m)$ on $\Omega$ such that $(\nu^n,m^n)$ converges weak$^*$ to $(\nu,m)$.
\end{proposition}

\subsection{The model}
We consider the growth-fragmentation equation under a general form:
\begin{equation}\label{eq:growthfragIntro}
\begin{aligned}
    \partial_t n(t,x) + \partial_x(g(x)n(t,x)) + B(x)n(t,x) &= \int_x^\infty k(x,y)B(y)n(t,y)\ \dy,\\
    g(0)n(t,0) &=0,\\
    n(0,x) & = n^0(x).
\end{aligned}    
\end{equation}
We  assume $n^0\in\Mcal^{+}(\R_+)$.

\medskip
The fundamental tool in studying the long-time asymptotics with the GRE method is the existence and uniqueness of the first eigenelements $(\lambda, N, \phi)$, i.e. solutions to the following primal and dual eigenproblems.
\begin{equation}\label{eq:primaleproblem}
    \begin{aligned}
    \frac{\partial}{\partial x}(g(x)N(x))+(B(x)+\lambda)N(x) &= \int_x^\infty k(x,y)B(y)N(y)\ \dy\\
    &\hspace{-4cm}g(0)N(0)=0,\;\;\; N(x)>0,\;\;\text{for}\;\; x>0,\;\;\;\int_0^\infty N(x)\dx=1,
    \end{aligned}    
    \end{equation}
   
    \begin{equation}\label{eq:dualeproblem}
    \begin{aligned}
    -g(x)\frac{\partial}{\partial x}(\phi(x))+(B(x)+\lambda)\phi(x) &= B(x)\int_0^x k(y,x)\phi(y)\ \dy\\
    &\hspace{-4cm}\; \phi(x)>0,\;\;\;\int_0^\infty \phi(x)N(x)\dx=1.
    \end{aligned}    
    \end{equation} 

We make the following assumptions on the parameters of the model.
\begin{align}
&B\in W^{1,\infty}(\R_+,\R_+^*),\qquad g\in W^{1,\infty}(\R_+,\R_+^*),\qquad \forall\; x\geq 0,\; g\geq g_0>0,\label{as:B}\\ 
&k\in {\Ccal}_b (\R_+\times\R_+), \qquad \int_0^y k(x,y)\dx=2,\qquad \int_0^y xk(x,y)\dx=y,\label{as:k}\\
&k(x,y<x)=0,\qquad k(x,y>x)>0. \label{as:k2}
\end{align}

\noindent These guarantee in particular existence and uniqueness of a solution $n\in\Ccal(\R_+;L^1_\phi(\R_+))$ for $L^1$ initial data (see e.g.~\cite{Perthame02}), existence of a unique measure solution for data in $\Mcal^+(\R_+)$ (cf. \cite{CarrilloGwiazda_11}), as well as existence and uniqueness of a dominant eigentriplet $(\lambda >0,N(x),\phi(x))$, cf.~\cite{DG}. In particular the functions $N$ and $\phi$ are continuous, $N$ is bounded and $\phi$ has at most polynomial growth. 
In what follows $N$ and $\phi$ will always denote the solutions to problems~\eqref{eq:primaleproblem} and~\eqref{eq:dualeproblem}, respectively. Let us remark that in the $L^1$ setting we have the following conservation law
\begin{equation}\label{eq:L1conservation}
    \intR n_\eps(t,x)e^{-\lambda t}\phi(x)\dx = \intR n^0(x)\phi(x)\dx.
\end{equation}

\subsection{Measure and measure-valued solutions}
Let us observe that there are two basic ways to treat the above model in the measure setting. \comMarie{The first one} is to consider a {\em measure solution}, i.e. a narrowly continuous map $t\mapsto\mu_t\in\Mcal^+(\R_+)$, which satisfies~\eqref{eq:growthfragIntro} in the weak sense, i.e. for each $\psi\in\Ccal^1_c(\R_+\times\R_+)$
\begin{equation}\label{eq:weakmeasuresol}
\begin{aligned}
    -\intR\intR&(\partial_t\psi(t,x)+\partial_x\psi(t,x)g(x))d\mu_t(x)\dt + \intR\intR\psi(t,x)B(x)d\mu_t(x)\dt \\
    &= \intR\intR\psi(t,x)\int_x^\infty k(x,y)B(y)d\mu_t(y)\dx\dt + \intR\psi(0,x) dn^0(x).
\end{aligned}
\end{equation}

Thus a measure solution is a family of time-parameterised non-negative Radon measures on the structure-physical domain $\R_+$. 

\comMarie{The second way is to work with generalised Young measures and corresponding measure-valued solutions.} To prove the generalised relative entropy inequality, which relies on considering a family of non-linear renormalisations of the equation, we choose to work \comMarie{in this second framework.}

A {\em measure-valued solution} is a generalised Young measure $(\nu,m)$, where the oscillation measure is a family of parameterised probabilities over the state domain $\R_+$ such that equation~\eqref{eq:growthfragIntro} is satisfied by its barycenters $\langle\nu_{t,x},\xi\rangle$, i.e. the following equation
\begin{equation}\label{eq:mvSol}
\begin{aligned}
    \partial_t&\left(\langle\nu_{t,x},\xi\rangle + m\right) + \partial_x\left(g(x)(\langle\nu_{t,x},\xi\rangle + m)\right) + B(x)(\langle\nu_{t,x},\xi\rangle + m) \\
    &= \int_x^\infty k(x,y)B(y)\langle\nu_{t,x},\xi\rangle\dy + \int_x^\infty k(x,y)B(y)dm(y)
\end{aligned}
\end{equation}
holds in the sense of distributions on $\R^*_+\times\R^*_+$.

It is proven in \cite{GLM} that equation~\eqref{eq:growthfragIntro} has a unique measure solution. To this solution one can associate a measure-valued solution - for example, given a measure solution $t\mapsto\mu_t$ one can define a measure-valued solution by
\[
\langle\delta_{\left\{\frac{d\mu_t^a}{d\Lcal^1}\right\}},\id\rangle = \mu_t^a,\qquad m = \mu_t^s
\]
where $\frac{d\mu_1}{d\mu_2}$ denotes the Radon-Nikodym derivative of $\mu_1$ with respect to $\mu_2$.

However, clearly, the measure-valued solutions are not unique - since the equation is linear, there is freedom in choosing the Young measure as long as the barycenter satisfies equation~\eqref{eq:mvSol}. For example, a different measure-valued solution can be defined by
\[
\langle\frac12\delta_{\left\{2\frac{d\mu_t^a}{d\Lcal^1}\right\}}+\frac12\delta_{\{0\}},\id\rangle = \mu_t^a.
\]

Uniqueness of measure-valued solution can be ensured by requiring that the generalised Young measure satisfies not only the equation, but also a family of nonlinear renormalisations. This was the case in the work of DiPerna \cite{DiPerna}, see also \cite{DebiecRelEntroSurvey}.

To establish the GRE inequality which will then be used to prove an asymptotic convergence result, we  consider the measure-valued solution generated by a sequence of regularized solutions. This  allows us to use the classical GRE method established in \cite{Pe2007}. Careful passage to the limit will then show that analogous inequalities hold for the measure-valued solution.

\section{GRE inequality}\label{sec3}

In this section we formulate and prove the generalised relative entropy inequality, our main tool in the study of long-time asymptotics for equation~\eqref{eq:growthfragIntro}. We  take advantage of the well-known GRE inequalities in the $L^1$ setting. To do so we consider the growth-fragmentation equation~\eqref{eq:growthfragIntro} for a sequence of regularized data and prove that we can pass to the limit, thus obtaining the desired inequalities in the measure setting.

Let $n_\ep^0\in L^1_\phi(\R_+)$ be a sequence of regularizations of $n^0$ converging weak$^*$ to $n^0$ in the space of measures and such that $\TV(n_\ep^0)\to\TV{(n^0)}$. 
Let $n_\ep$ denote the corresponding unique solution to~\eqref{eq:growthfragIntro} with $n_\ep^0$ as an initial condition. Then for any differentiable strictly convex admissible integrand $H$ we define the usual relative entropy 
	\[
	\Hcal_\ep (t):=\intR \phi(x)N(x) H\left(\f{n_\ep(t,x)e^{-\lambda t}}{N(x)}\right) \dx
	\]
	and entropy dissipation
    \begin{equation*}
        \begin{aligned}
        D^H_\ep(t)=\intR\intR &\phi(x) N(y) B(y) k(x,y)\Biggl\{ H\left(\f{n_\ep(t,y) e^{-\lambda t}}{N(y)}\right)- H\left(\f{n_\ep(t,x)e^{-\lambda t}}{N(x)}\right)\\
        &- H'\left(\f{n_\ep(t,x)e^{-\lambda t}}{N(x)}\right)\left[\f{n_\ep(t,y)e^{-\lambda t}}{N(y)} - \f{n_\ep(t,x)e^{-\lambda t}}{N(x)}\right]\Biggr\}\ \dx\dy.
        \end{aligned}
    \end{equation*}
Then, as shown e.g. in~\cite{MMP1}, one can show that
\begin{equation}\label{eq:classicalGRE1}
\f{d}{dt}\left\{\intR \phi(x) N(x) H\left(\f{n_\ep(t,x)e^{-\lambda t} }{N(x)}\right)\ \dx\right\} =-D^H_\ep(t)
\end{equation}
with the right-hand side being non-positive due to convexity of $H$. Hence the relative entropy is non-increasing. It follows that $\Hcal_\eps(t)\leq\Hcal_\eps(0)$ and, since $H\geq 0$,
\begin{equation}\label{eq:classicalGRE2}
\intR D_\ep^H(t)\ \dt \leq \Hcal_\ep(0).
\end{equation}

In the next proposition we prove corresponding inequalities for the measure-valued solution generated by the sequence $n_\eps$. This result is an analogue of Theorem~5.1 in \cite{GW}.

\begin{proposition}\label{prop:GRE}
With notation as above,
	there exists a subsequence (not relabelled), generating a generalised Young measure $(\nu,m)$ with $m=m_t\otimes \dt$ for a family of positive Radon measures $m_t$, such that 
	\begin{equation}\label{eq:GRE}
	\begin{aligned}
	\lim\limits_{\ep\to 0}\;\intR\chi(t)\Hcal_\ep (t)\ \dt =
	\int_0^\infty\chi(t)\;\biggl(\intR\phi(x)N(x)&\langle\nu_{t,x}(\alpha), H(\alpha)\rangle\dx\\
	&+ \intR \phi(x)N(x)H^\infty dm_t(x)\biggr)\ \dt
	\end{aligned}
	\end{equation}
	for any $\chi\in C_c([0,\infty))$, and
	\begin{equation}\label{eq:dissipGRE}
	\begin{aligned}
	\lim\limits_{\ep\to 0}\int_0^\infty D^H_\ep (t)\ \dt &= \\ &\hspace{-3cm}\intR\intR\intR  \phi(x)N(y)B(y)k(x,y)\langle\nu_{t,y}(\xi)\otimes\nu_{t,x}(\alpha),H(\xi)-H(\alpha)-H'(\alpha)(\xi-\alpha)\rangle\dx \dy \dt\\
	&\hspace{-2cm}+\intR\intR\intR\phi(x)N(y)B(y)k(x,y)\langle\nu_{t,x}(\alpha),H^\infty - H'(\alpha)\rangle dm_t(y)\dx\dt \geq 0.
	\end{aligned}
	\end{equation}
	We denote the limits on the left-hand sides of the above equations by $\intR\chi(t)\Hcal(t)\ \dt$ and $\intR D^H(t)\ \dt$, respectively, thus defining the measure-valued relative entropy and entropy dissipation for almost every $t$. We further set
	\begin{equation}\label{eq:dissip0}
	    \Hcal(0)\coloneqq\intR\phi(x)N(x)H\left(\frac{(n^0)^a(x)}{N(x)}\right)\dx + \intR\phi(x)H^\infty\left(\frac{(n^0)^s}{|(n^0)^s|}(x)\right)d|(n^0)^s(x)|.
	\end{equation}
	We then have
	\begin{equation}\label{eq:dissipGRE:measure}
	\frac{d}{dt}\Hcal (t) \leq 0 \quad \text{in the sense of distributions}, 
	\end{equation}
and
	\begin{equation}\label{eq:entropydissipationinL1}
	\int_0^\infty D^H (t) dt \leq \Hcal (0).
	\end{equation}
\end{proposition}
\medskip
\begin{proof}
The function $t\mapsto\intR n_\eps(t,x)e^{-\lambda t}\phi(x)\dx$ is constant and the function $N$ is strictly positive on $(0,\infty)$. Therefore the sequence $u_\eps(t,x)\coloneqq\frac{n_\eps(t,x)e^{-\lambda t}}{N(x)}$ is uniformly bounded in $L^\infty(\R_+;L^1_{\phi,\loc}(\R_+))$. Hence we can apply Proposition~\ref{prop:Alibert} to obtain a generalised Young measure $(\nu,m)$ on $\R_+\times\R_+$. 


Since $u_\eps \in L^\infty(\R_+;L^1_{\phi,\loc}(\R_+))$, we have $m\in L^\infty(\R_+;\Mcal(\R_+;\phi))$.
By a standard disintegration argument (see for instance \cite[Theorem~1.5.1]{Evans}) we can write the slicing measure for $m$,
$m(\dt,\dx) = m_t(\dx)\otimes\dt$, where the map $t\mapsto m_t\in\Mcal^+(\R_+;\phi)$ is measurable and bounded.

By Proposition~\ref{prop:Alibert} we have the weak$^*$ convergence
\begin{equation*}
H(u_\eps(t,x))(\dt\otimes\phi(x)\dx)\;\stackrel{*}{\rightharpoonup}\; \langle\nu_{t,x},H\rangle(\dt\otimes\phi(x)\dx) + H^\infty m.
\end{equation*}
Testing with $(t,x)\mapsto\chi(t)N(x)$ where $\chi\in\Ccal_c(\R_+)$, we obtain~\eqref{eq:GRE}. 
Further, the convergence $\intR\chi(t)\Hcal_\eps(t)\dt\to\intR\chi(t)\Hcal_\eps(t)\dt$ implies~\eqref{eq:dissipGRE:measure}, since for $\Hcal_\eps$ we have the corresponding inequality~\eqref{eq:classicalGRE1}. 

\medskip
We  now investigate the limit as $\eps\to0$ of $\intR D_\eps^H(t)\dt$. Denoting $\Phi(x,y)\coloneqq k(x,y)N(y)B(y)$ we have
\begin{equation}\label{eq:entropydissipation}
\begin{aligned}
    D_\eps^H(t) = \intR\intR\Phi(x,y)\phi(x)[H(u_\eps(t,y))-H(u_\eps(t,x))&-H'(u_\eps(t,x))u_\eps(t,y)\\
    &+ H'(u_\eps(t,x))u_\eps(t,x)]\dx\dy.
\end{aligned}
\end{equation}
We consider each of the four terms of the sum separately on the restricted domain $[0,T]\times[\eta,K]^2$ for fixed $T>0$ and $K>\eta>0$. 
Let $D_{\eps,
\eta,K}^H$ denote the entropy dissipation with the integrals of~\eqref{eq:entropydissipation} each taken over the subsets $[\eta,K]$ of $\R_+$.

We now apply Proposition~\ref{prop:Alibert} to the sequence $u_\eps$, the measure $\dt\otimes\phi(x)\dx$ on the set $[0,T]\times [\eta,K]$. The first two and the last integrands of $D_{\eps,\eta,K}^H(t)$ depend on $t$ and only either on $x$ or on $y$. Therefore we can pass to the limit as $\eps\to0$ by Proposition~\ref{prop:Alibert} using a convenient test function. More precisely, testing with
$(t,x)\mapsto\int_\eta^K\Phi(x,y)\dy$, we obtain the convergence
\begin{equation*}
    \begin{aligned}
    -\int_0^T\int_\eta^K\int_\eta^K\Phi(x,y)\phi(x)H(u_\eps(t,x))\dy\dx\dt\longrightarrow&-\int_0^T\int_\eta^K\int_\eta^K\Phi(x,y)\phi(x)\langle\nu_{t,x},H\rangle\dy\dx\dt\\
    &-\int_0^T\int_\eta^K\int_\eta^K\Phi(x,y)\phi(x)H^\infty dm_t(x)\dy\dt,
    \end{aligned}
\end{equation*}
and, noticing that the recession function of $\alpha\mapsto\alpha H'(\alpha)$ is $H^\infty$,
\begin{equation*}
    \begin{aligned}
    &\int_0^T\int_\eta^K\int_\eta^K\Phi(x,y)\phi(x)H'(u_\eps(t,x))u_\eps(t,x)\dy\dx\dt\longrightarrow\\
    &\int_0^T\int_\eta^K\int_\eta^K\Phi(x,y)\phi(x)\langle\nu_{t,x},\alpha H'(\alpha)\rangle\dy\dx\dt
    +\int_0^T\int_\eta^K\int_\eta^K\Phi(x,y)\phi(x)H^\infty dm_t(x)\dy\dt.
    \end{aligned}
\end{equation*}
Likewise, using $(t,y)\mapsto\frac{1}{\phi(y)}\int_\eta^K\Phi(x,y)\phi(x)\dx$, we obtain
\begin{equation*}
    \begin{aligned}
    \int_0^T\int_\eta^K\int_\eta^K\Phi(x,y)\phi(x)H(u_\eps(t,y))\dx\dy\dt\rightarrow&\int_0^T\int_\eta^K\int_\eta^K\Phi(x,y)\phi(x)\langle\nu_{t,y},H\rangle\dx\dy\dt\\
    &+\int_0^T\int_\eta^K\int_\eta^K\Phi(x,y)\phi(x)H^\infty dm_t(y)\dx\dt.
    \end{aligned}
\end{equation*}

There remains the term of $D_{\eps,\eta,K}^H$ in which the dependence on $u_\eps$ combines $x$ and $y$. To deal with this term we separate variables by testing against functions of the form $f_1(x)f_2(y)$. We then consider
\begin{equation*}
    \begin{aligned}
    -\int_0^T\iint_{[\eta,K]^2}f_1(x)f_2(y)H'(u_\eps(t,x))u_\eps(t,y)\dx\dy\dt &\\
    &\hspace{-5cm}=-\int_0^T\left(\int_\eta^K f_1(x)H'(u_\eps(t,x))\dx\right)\left(\int_\eta^K f_2(y)u_\eps(t,y)\dy\right)\dt.
    \end{aligned}
\end{equation*}
The integrands are now split, one containing the $x$ dependence, and one the $y$ dependence. However, extra care is required here to pass to the limit. As the Young measures depend both on time and space, it is possible for the oscillations to appear in both directions. We therefore require appropriate time regularity of at least one of the sequences to guarantee the desired behaviour of the limit of the product.

Such requirement is met by noticing that since $u_\eps\in\Ccal([0,T];L^1_\phi(\R_+))$ uniformly, we have $u_\eps$ uniformly in $W^{1,\infty}([0,T];(\Mcal^+(\mathcal{\R_+}),\norm{\cdot}_{(W^{1,\infty})^*}))$, cf.~\cite[Lemma~4.1]{CarrilloGwiazda_11}. Assuming $f_2\in W^{1,\infty}(\R_+)$ we therefore have
\[
\left(t\mapsto\int_\eta^K f_2(y)u_\eps(t,y)\dy\right)\in W^{1,\infty}([0,T]).
\]
\noindent This in turn implies strong convergence of $\int_\eta^K f_2(y)u_\eps(t,y)\dy$ in $\Ccal([0,T])$, by virtue of Arz\'{e}la-Ascoli theorem.
Therefore we have (noting that $(H')^\infty\equiv0$ by sublinear growth of $H$)
\begin{equation*}
    \begin{aligned}
        -\int_0^T\iint_{[\eta,K]^2}f_1(x)f_2(y)H'(u_\eps(t,x))u_\eps(t,y)\dx\dy\dt &\\
    &\hspace{-7cm}=-\int_0^T\left(\int_\eta^K f_1(x)H'(u_\eps(t,x))\dx\right)\left(\int_\eta^K f_2(y)u_\eps(t,y)\dy\right)\dt\\
    &\hspace{-7cm}\longrightarrow-\int_0^T\left(\int_\eta^K f_1(x)\langle\nu_{t,x},H'\rangle\dx\right)\left(\int_\eta^K f_2(y)\langle\nu_{t,y},\id\rangle\dy\right)\dt\\
    &\hspace{-6cm}-\int_0^T\left(\int_\eta^K f_1(x)\langle\nu_{t,x},H'\rangle\dx\right)\left(\int_\eta^K f_2(y)dm_t(y)\right)\dt\\
    &\hspace{-7cm}=-\int_0^T\iint_{[\eta,K]^2}f_1(x)f_2(y)\langle\nu_{t,x},H'(\alpha)\rangle\langle\nu_{t,y},\xi\rangle\dx\dy \\
    &\hspace{-6cm}-\int_0^T\iint_{[\eta,K]^2}f_1(x)f_2(y)\langle\nu_{t,x},H'(\alpha)\rangle dm_t(y)\dx\dt.
    \end{aligned}
\end{equation*}
By density of the linear space spanned by separable functions in the space of bounded continuous functions of $(x,y)$ we obtain
\begin{equation*}
    \begin{aligned}
    -\int_0^T\iint_{[\eta,K]^2}&\Phi(x,y)\phi(x)H'(u_\eps(t,x))u_\eps(t,y)\dx\dy\dt\\
    &\longrightarrow \int_0^T\iint_{[\eta,K]^2}\Phi(x,y)\phi(x)\langle\nu_{t,x},H'(\alpha)\rangle\langle\nu_{t,y},\xi\rangle\dx\dy\dt\\
    &\hspace{1cm}-\int_0^T\iint_{[\eta,K]^2}\Phi(x,y)\phi(x)\langle\nu_{t,x},H'(\alpha)\rangle dm_t(y)\dx\dt.
    \end{aligned}
\end{equation*}
Gathering all the terms we thus obtain the convergence as $\eps\to0$
\begin{equation*}
    \int_0^TD_{\eps,\eta,K}^H(t)\dt\longrightarrow\int_0^TD_{\eta,K}^H(t)\dt
\end{equation*}
with
\begin{equation*}
    \begin{aligned}
    D_{\eta,K}^H(t)\coloneqq&\iint_{[\eta,K]^2}\Phi(x,y)\phi(x)\langle\nu_{t,y}(\xi)\otimes\nu_{t,x}(\alpha), H(\xi)-H(\alpha)-H'(\alpha)(\xi-\alpha)\rangle\dx\dy\\
    &+\iint_{[\eta,K]^2}\Phi(x,y)\phi(x)\langle\nu_{t,x}(\alpha),H^\infty-H'(\alpha)\rangle dm_t(y)\dx.
    \end{aligned}
\end{equation*}
Observe that since $\Phi$ is non-negative and $H$ is convex, the integrand of $D_{\eps,\eta,K}^H$ is non-negative. Hence so is the integrand of the limit. Therefore, by Monotone Convergence, we can pass to the limit $\eta\to0$, $K\to\infty$, and $T\to\infty$ to obtain
	\begin{equation*}
	\begin{aligned}
	0\leq\lim\limits_{\ep\to 0}\;&\int_0^\infty D^H_\ep (t)\ \dt = \\ \intR\intR\intR & \phi(x)N(y)B(y)k(x,y)\langle\nu_{t,y}(\xi)\otimes\nu_{t,x}(\alpha),H(\xi)-H(\alpha)-H'(\alpha)(\xi-\alpha)\rangle\dx \dy \dt\\
	&\hspace{-1cm}+\intR\intR\intR\phi(x)N(y)B(y)k(x,y)\langle\nu_{t,x}(\alpha),H^\infty - H'(\alpha)\rangle dm_t(y)\dx\dt.
	\end{aligned}
	\end{equation*}	

Finally we note that by the Reshetnyak continuity theorem, cf.~\cite{GW, KristensenRindler2010} we have the convergence $\Hcal_\ep(0)\to\Hcal(0)$. Together with~\eqref{eq:classicalGRE2} this implies~\eqref{eq:dissipGRE:measure}. 
\end{proof}

\section{Long-time asymptotics}\label{sec4}
In this section we use the result of the previous section to prove that a measure-valued solution of~\eqref{eq:growthfragIntro} converges to the steady-state solution. More precisely we prove
\begin{theorem}
Let $n^0\in\Mcal(\R_+)$ and let $n$ solve the growth-fragmentation equation~\eqref{eq:growthfragIntro}. Then
\begin{equation}\label{eq:long-time}
\lim\limits_{t\to\infty}\intR\phi(x)\mathrm{d}|n(t,x)-m_0N(x)\mathcal{L}^1| = 0
\end{equation}
where $m_0\coloneqq\intR\phi(x)\mathrm{d}n^0(x)$ and $\mathcal{L}^1$ denotes the $1$-dimensional Lebesgue measure.
\end{theorem}
\begin{proof}
From inequality~\eqref{eq:entropydissipationinL1} we see that $D^H$ belongs to $L^1(\R_+)$. Therefore there exists a sequence of times $t_n\to\infty$ such that
\[
\lim\limits_{n\to\infty}D\comMarie{^H}(t_n) = 0.
\]
Consider the corresponding sequence of generalised Young measures $(\nu_{t_n,x},m_{t_n})$. Thanks to the inequality $\Hcal(t)\leq\Hcal(0)$ this sequence is uniformly bounded in the sense that
\begin{equation}
    \sup\limits_{n}\left\{\intR\phi(x)N(x)\langle\nu_{t,x}(\alpha),|\alpha|\rangle\dx + \intR\phi(x)N(x)dm_{t_n}(x)\right\}\;<\;\infty.
\end{equation}
Therefore by the compactness property of Proposition~\ref{prop:KristenRindler} there is a subsequence, not relabelled, and a generalised Young measure $(\bar\nu_x,\bar m)$ such that
\[
(\nu_{t_n,x},m_{t_n})\stackrel{\ast}{\rightharpoonup}(\bar\nu_x,\bar m)
\]
in the sense of measures.
We  now show that the corresponding "entropy dissipation"
\begin{equation}\label{Dinfinity}
\begin{aligned}
D_\infty^H\coloneqq\intR\intR&\Phi(x,y)\phi(x)\langle\bar\nu_y(\xi)\otimes\bar\nu_x(\alpha),H(\xi)-H(\alpha)-H'(\alpha)(\xi-\alpha)\rangle\dx\dy\\
&+\intR\intR\Phi(x,y)\phi(x)\langle\bar\nu_x(\alpha),H^\infty-H'(\alpha)\rangle d\bar m(y)\dx
\end{aligned}
\end{equation}
is zero. To this end we  argue that
\[
D_\infty^H = \lim\limits_{n\to\infty}D^H(t_n).
\]
Indeed this follows by the same arguments as in the proof of Proposition~\ref{prop:GRE}. In fact now the "mixed" term poses no additional difficulty as there is no time integral.
It therefore follows that 
\begin{equation}\label{zeroDissipation}
D_\infty^H = 0.
\end{equation}
As $H$ is convex, both integrands in~\eqref{Dinfinity} are non-negative.
Therefore~\eqref{zeroDissipation} implies that both the integrals of $D_\infty^H$ are zero. In particular
\[
\intR\intR H(\xi)-H(\alpha)-H'(\alpha)(\xi-\alpha)d\bar\nu_x(\alpha)d\bar\nu_y(\xi) = 0,
\]
and since the integrand vanishes if and only if $\xi=\alpha$, this implies that the Young measure $\bar\nu$ is a Dirac measure concentrated at a constant. Then the vanishing of the second integral of $D_\infty^H$ implies that $\bar m = 0$. Moreover, the constant can be identified as
\begin{equation}\label{homogeneousDirac}
    m_0\coloneqq\intR\phi(x)dn^0(x)
\end{equation}
by virtue of the conservation in time of
\[
\intR\phi(x)e^{-\lambda t}\langle\nu_{t,x},\cdot\rangle\dx + \intR\phi(x)e^{-\lambda t}dm_t(x).
\]
By virtue of Proposition~\ref{prop:Alibert} with $H=|\cdot-m_0|$ it then follows that
\begin{equation*}
    \lim\limits_{n\to\infty}\;\intR\phi(x)d|n(t_n,x)e^{-\lambda t_n}-m_0N(x)\Lcal^1| = 0,
\end{equation*}
which is the desired result, at least for our particular sequence of times.

Finally, we can argue that the last convergence holds for the entire time limit $t\to\infty$, invoking the monotonicity of the relative entropy $\Hcal$. Indeed, the choice $H=|\cdot-m_0|$ in~\eqref{eq:dissip0} yields the monotonicity in time of
\[
\intR\phi(x)d|n(t,x)e^{-\lambda t}-m_0N(x)\Lcal^1|,
\]
and the result follows.
\end{proof}

\section*{Conclusion}

\comMarie{In this article, we have proved the long-time convergence of measure-valued solutions to the growth-fragmentation equation. This result extends previously obtained results for $L^1_\phi$ solutions~\cite{MMP2}. As for the renewal equation~\cite{GW}, it is based on extending the generalised relative entropy inequality to measure-valued solutions, thanks to recession functions. Generalised Young measures provide an adequate framework to represent the measure-valued solutions and their entropy functionals.}

\comMarie{Under slightly stronger assumptions on the fragmentation kernel $k$, e.g. the ones assumed in~\cite{CCM}, it has been proved  that an entropy-entropy dissipation inequality could be obtained. Under such assumptions, we could obtain in a simple way a stronger result of exponential convergence, see the proof of Theorem~4.1. in~\cite{GW}. However the aboveseen method allows us to extend the convergence to spaces where no spectral gap exists~\cite{bernard:hal-01313817}.}

\comMarie{A specially interesting case of application of this method would be critical cases where the dominant eigenvalue is not unique but is given by a countable set of eigenvalues. It has been proved that for $L^2$ initial conditions, the solution then converges to its projection on the  space spanned by the dominant eigensolutions~\cite{bernard2016cyclic}. In the case of measure-valued initial condition, due to the fact that the equation has not anymore a regularisation effect, the asymptotic limit is expected to be the periodically oscillating measure, projection of the initial condition on the space of measures spanned by the dominant eigensolutions. This is a subject for future work.}

\comTomek{{\bf Acknowledgements.}
T.~D. would like to thank the Institute for Applied Mathematics of the Leibniz University of Hannover for its warm hospitality during his stay, when part of this work was completed. 
\\
This work was partially supported by the Simons - Foundation grant 346300 and the Polish Government MNiSW 2015-2019 matching fund.
The research of T.~D. was supported by National Science Center (Poland) 2014/13/B/ST1/03094. M.D.'s research was supported by the Wolfgang Pauli Institute (Vienna) and the ERC Starting Grant SKIPPER$^{AD}$ (number 306321). P.~G. received support from National Science Center (Poland) 2015/18/M/ST1/00075.}

\bibliographystyle{plain}       
\bibliography{TDebiec_bib}           

\begin{thebibliography}{10}

\bibitem{AlibertBouchitte}
J.~J. Alibert and G.~Bouchitt\'e.
\newblock Non-uniform integrability and generalized {Y}oung measures.
\newblock {\em J. Convex Anal.}, 4(1):129--147, 1997.

\bibitem{ball}
J.~M. Ball.
\newblock A version of the fundamental theorem for {Y}oung measures.
\newblock In {\em P{DE}s and continuum models of phase transitions ({N}ice,
  1988)}, volume 344 of {\em Lecture Notes in Phys.}, pages 207--215. Springer,
  Berlin, 1989.

\bibitem{BellAnderson}
George~I. Bell and Ernest~C. Anderson.
\newblock Cell growth and division: I. a mathematical model with applications
  to cell volume distributions in mammalian suspension cultures.
\newblock {\em Biophysical Journal}, 7(4):329 -- 351, 1967.

\bibitem{bernard2016cyclic}
Etienne Bernard, Marie Doumic, and Pierre Gabriel.
\newblock Cyclic asymptotic behaviour of a population reproducing by fission
  into two equal parts.
\newblock {\em arXiv preprint arXiv:1609.03846}, 2016.

\bibitem{bernard:hal-01313817}
Etienne Bernard and Pierre Gabriel.
\newblock {Asymptotic behavior of the growth-fragmentation equation with
  bounded fragmentation rate}.
\newblock working paper or preprint, May 2016.

\bibitem{BrDeLeSz2011}
Yann Brenier, Camillo De~Lellis, and L{{\'a}}szl{{\'o}} Sz{{\'e}}kelyhidi, Jr.
\newblock Weak-strong uniqueness for measure-valued solutions.
\newblock {\em Comm. Math. Phys.}, 305(2):351--361, 2011.

\bibitem{CCM}
M.~J. C\'aceres, J.~A. Ca\~nizo, and S.~Mischler.
\newblock Rate of convergence to the remarkable state for fragmentation and
  growth-fragmentation equations.
\newblock {\em J. Math. Pures Appl.}, 96(4):334--362, 2011.

\bibitem{CarrilloGwiazda_11}
J.A. Carrillo, R.M. Colombo, P.~Gwiazda, and A.~Ulikowska.
\newblock Structured populations, cell growth and measure valued balance laws.
\newblock {\em Journal of Differential Equations}, 252(4):3245 -- 3277, 2012.

\bibitem{Christoforou2017}
C.~Christoforou and A.~E. Tzavaras.
\newblock Relative entropy for hyperbolic--parabolic systems and application to
  the constitutive theory of thermoviscoelasticity.
\newblock {\em Arch. Rational Mech. Anal.}, Dec 2017.

\bibitem{Dafermos1979}
C.~M. Dafermos.
\newblock The second law of thermodynamics and stability.
\newblock {\em Arch. Rational Mech. Anal.}, 70(2):167--179, 1979.

\bibitem{DebiecRelEntroSurvey}
T.~D{\k e}biec, P.~Gwiazda, K.~{\L}yczek, and A.~{\'S}wierczewska-Gwiazda.
\newblock Relative entropy method for measure-valued solutions in natural
  sciences.
\newblock {\em To appear in Topol. Meth. Nonlin. Anal.}, 2017.

\bibitem{tzavaras1}
S.~Demoulini, D.~M.~A. Stuart, and A.~E. Tzavaras.
\newblock Weak-strong uniqueness of dissipative measure-valued solutions for
  polyconvex elastodynamics.
\newblock {\em Arch. Rational Mech. Anal.}, 205(3):927--961, 2012.

\bibitem{DiPerna}
R.~J. DiPerna.
\newblock Measure-valued solutions to conservation laws.
\newblock {\em Arch. Rational Mech. Anal.}, 88(3):223--270, 1985.

\bibitem{DiPernaMajda}
R.~J. DiPerna and A.~J. Majda.
\newblock Oscillations and concentrations in weak solutions of the
  incompressible fluid equations.
\newblock {\em Comm. Math. Phys.}, 108(4):667--689, 1987.

\bibitem{DG}
M.~Doumic and P.~Gabriel.
\newblock Eigenelements of a general aggregation-fragmentation model.
\newblock {\em Math. Models and Methods Appl. Sci.}, 20(5):757, 2009.

\bibitem{Evans}
L.~C. Evans.
\newblock {\em Weak convergence methods for nonlinear partial differential
  equations}.
\newblock Number~74. American Mathematical Soc., 1990.

\bibitem{FGSW2016}
Eduard Feireisl, Piotr Gwiazda, Agnieszka \'{S}wierczewska Gwiazda, and Emil
  Wiedemann.
\newblock Dissipative measure-valued solutions to the compressible
  {N}avier-{S}tokes system.
\newblock {\em Calc. Var. Partial Differential Equations}, 55(6):Art. 141, 20,
  2016.

\bibitem{Gabriel:2017vl}
P.~Gabriel.
\newblock Measure solutions to the conservative renewal equation.
\newblock {\em Preprint}, 2017.

\bibitem{GwiazdaKreml2018}
P.~{Gwiazda}, O.~{Kreml}, and A.~{{\'S}wierczewska-Gwiazda}.
\newblock {Dissipative measure valued solutions for general conservation laws}.
\newblock {\em ArXiv e-prints}, January 2018.

\bibitem{GLM}
P.~Gwiazda, T.~Lorenz, and A.~Marciniak-Czochra.
\newblock A nonlinear structured population model: Lipshitz continuity of
  measure valued solutions with respect to model ingredients.
\newblock {\em J. Differential Equations}, 248:2703--2735, 2010.

\bibitem{GW}
P.~Gwiazda and E.~Wiedemann.
\newblock Generalized entropy method for the renewal equation with measure
  data.
\newblock {\em Commun. Math. Sci.}, 15(2):577--586, 2016.

\bibitem{GSGW}
Piotr Gwiazda, Agnieszka {\'S}wierczewska-Gwiazda, and Emil Wiedemann.
\newblock Weak-strong uniqueness for measure-valued solutions of some
  compressible fluid models.
\newblock {\em Nonlinearity}, 28(11):3873, 2015.

\bibitem{Kermack1}
W.~O. Kermack and A.~G. McKendrick.
\newblock A contribution to the mathematical theory of epidemics.
\newblock {\em Proc. Roy. Society of London, Series A}, 115(772):700--721,
  1927.

\bibitem{Kermack2}
W.~O. Kermack and A.~G. McKendrick.
\newblock Contribution to the mathematical theory of epidemics. ii. the problem
  of endemicity.
\newblock {\em Proc. Roy. Society of London, Series A}, 138(834):55--83, 1932.

\bibitem{KristensenRindler2010}
J.~Kristensen and F.~Rindler.
\newblock Relaxation of signed integral functionals in {$BV$}.
\newblock {\em Calc. Var. Partial Differential Equations}, 37(1-2):29--62,
  2010.

\bibitem{KristensenRindler2012}
J.~Kristensen and F.~Rindler.
\newblock Characterization of {G}eneralized {G}radient {Y}oung {M}easures
  {G}enerated by {S}equences in {$W^{1,1}$} and {$BV$}.
\newblock {\em Arch. Rational Mech. Anal.}, 197(2):539--598, 2012.

\bibitem{M1}
P.~Michel.
\newblock Existence of a solution to the cell division eigenproblem.
\newblock {\em Math. Models Methods Appl. Sci.}, 16(7, suppl.):1125--1153,
  2006.

\bibitem{MMP1}
P.~Michel, S.~Mischler, and B.~Perthame.
\newblock General entropy equations for structured population models and
  scattering.
\newblock {\em C. R. Math. Acad. Sci. Paris}, 338(9):697--702, 2004.

\bibitem{MMP2}
P.~Michel, S.~Mischler, and B.~Perthame.
\newblock General relative entropy inequality: an illustration on growth
  models.
\newblock {\em J. Math. Pures Appl. (9)}, 84(9):1235--1260, 2005.

\bibitem{mischler:frag}
S.~Mischler and J.~Scher.
\newblock Spectral analysis of semigroups and growth-fragmentation equations.
\newblock {\em Ann. Inst. H. Poincar\'e Anal. Non Lin\'eaire}, 33(3):849--898,
  2016.

\bibitem{Perthame02}
B.~Perthame.
\newblock {\em Kinetic formulation of conservation laws}, volume~21 of {\em
  Oxford Lecture Series in Mathematics and its Applications}.
\newblock Oxford University Press, Oxford, 2002.

\bibitem{Pe2007}
B.~Perthame.
\newblock {\em Transport equations in biology}.
\newblock Frontiers in Mathematics. Birkh{\"a}user Verlag, Basel, 2007.

\bibitem{PR}
B.~Perthame and L.~Ryzhik.
\newblock Exponential decay for the fragmentation or cell-division equation.
\newblock {\em J. Differential Equations}, 210(1):155--177, 2005.

\bibitem{SharpeLotka}
F.~R. Sharpe and A.~J. Lotka.
\newblock A problem in age-distribution.
\newblock {\em Philosophical Magazine}, 21:435--438, 1911.

\bibitem{SinkoStreifer}
J.W. Sinko and W.~Streifer.
\newblock A new model for age-size structure of a population.
\newblock {\em Ecology}, 48(6):910--918, 1967.

\end{thebibliography}

\end{document}